\documentclass[11pt,reqno]{amsart}

\usepackage[utf8]{inputenc}
\usepackage[T1]{fontenc}
\usepackage{lmodern}
\usepackage{microtype} 

\usepackage{amsmath,amssymb,amsfonts,amsthm}
\usepackage{mathtools}
\usepackage{mathrsfs}

\usepackage{graphicx}
\usepackage{booktabs} 
\usepackage{array}

\usepackage[margin=1in]{geometry}

\theoremstyle{plain}
\newtheorem{theorem}{Theorem}[section]

\newtheorem{proposition}[theorem]{Proposition}

\theoremstyle{definition}
\newtheorem{definition}[theorem]{Definition}

\newtheorem{remark}[theorem]{Remark}

\numberwithin{equation}{section}

\usepackage{xcolor}
\definecolor{winered}{rgb}{0.5,0,0}
\definecolor{darkblue}{rgb}{0,0,0.6}

\usepackage[
    pdftex,
    colorlinks=true,
    linkcolor=darkblue,
    citecolor=winered,
    urlcolor=darkblue,
    hypertexnames=false 
]{hyperref}

\title[Transmutation operators for SLEIF with an integrable potential]{Transmutation operators for the Sturm--Liouville equation in impedance form with an integrable potential}

\author{Briceyda B. Delgado}
\address{INFOTEC, {Centro de Investigación e Innovación en Tecnologías de la Información  y Comunicación}, Cto. Tecnopolo Sur No. 112, Pocitos, Aguascalientes 20326, Mexico}
\email{briceyda.delgado@infotec.edu.mx}

\author{F. Ay\c{c}a \c{C}etinkaya}
\address{Department of Mathematics, University of Tennessee at Chattanooga, 37403 Chattanooga, TN, USA.}
\email{fatmaayca-cetinkaya@utc.edu}
\date{\today}

\begin{document}

\begin{abstract}
In this work, we study a family of transmutation operators for the Sturm--Liouville operator in impedance form with an integrable potential satisfying a smallness condition. Specifically, we consider transmutations associated with the pair
\[\textbf{L}_r=\frac{1}{r^2 (x)}\frac{d}{dx}r^2 (x)\frac{d}{dx} \quad \text{and} \quad  \textbf{M}=\frac{d^2}{dx^2}.
\]
We extend existing results by showing that a particular transmutation operator on $W^{3,1} (-a,a)$ can be represented in integral form, where the kernel function satisfies a corresponding Goursat problem. We establish the existence of this kernel using two distinct approaches: the classical method of successive approximations and an application of Picard's theorem.
\end{abstract}

\maketitle

\section{Introduction}
A nonzero operator $T$ that intertwines two operators $A$ and $B$, that is, transform $A$ into $B$ via the relation $AT=TB$ is called a transmutation operator. The essential utility of these operators lies in their capacity to establish a rigorous mapping between the solutions of a simplified model equation, typically involving a standard differential operator, and those of a more complex perturbed operator. The foundation of this theory traces back to the pioneering work of Delsarte \cite{Delsarte1938, DelsarteLions1957}, Gelfand and Levitan \cite{Levitan}, Marchenko \cite{Marchenko}, Faddeev \cite{Faddeev}, Carroll \cite{Caroll}, Begehr \cite{Begehr}, and Povzner \cite{povzner1948differential}. More recently, Kravchenko and Sitnik \cite{KravchenkoSitnik2020} provided a comprehensive overview of modern developments in utilizing transmutation operators to solve problems in mathematical physics. Further contemporary advancements regarding the application of these methods to both direct and inverse Sturm–Liouville problems on finite and infinite intervals can be found in \cite{KravchenkoBook, shishkina2020transmutations}. The scope of the transmutation operator method extends well beyond unperturbed second derivatives; it encompasses a broader class of equations, such as singular differential equations governed by the action of the Bessel operator. Classic examples include the classical and generalized Euler–Poisson–Darboux equations investigated by Shishkina and Sitnik \cite{shishkina2017general}, as well as wave-type systems treated via the integral transforms composition method by Fitouhi et al. \cite{ahmed2018applications}. While much of the historical literature relies on mapping back to standard derivatives or singular frameworks, recent advancements have directly targeted the alternative structures of mathematical physics, such as complete systems of solutions and Darboux transformations for Sturm--Liouville equations in impedance form \cite{vicente2026complete}. For a detailed and recent analysis of the transmutation operator associated to the Sturm--Liouville equations in impedance form see \cite{vicente2025transmutation}.

In this paper, we study transmutation operators for Sturm--Liouville equations with integrable coefficients, where the underlying differential equation in impedance form is given by
\begin{equation} \label{1.1-intro} 
    \dfrac{d}{dx}\bigg(p(x)\dfrac{du}{dx}\bigg)+\lambda p(x)u=0, \ x \in I \subset \mathbb{R}, \ \lambda \in \mathbb{C}.
\end{equation}

Equation~\eqref{1.1-intro} is of practical importance. For example, the free longitudinal vibrations of a thin straight rod with cross-sectional area $A(x)$, density $\varrho$ and Young's modulus $E$ are governed by 
\[
(A(x)w'(x))'+\lambda A(x)w(x)=0,
\]
where $\lambda=\varrho w^2 / E$ and $w$ denotes the vibration frequency. Similarly, the free torsional vibrations of a thin straight rod characterized by the second moment of area $J(x)$, density $\varrho$ and shear modulus $G$ satisfy
\[
    (J(x)\theta'(x))'+\lambda J(x)\theta (x)=0,
\]
where $\lambda =\varrho w^2 / G$.
As detailed in \cite[Chapter 10]{Gladwellbook}, there exists a precise one-to-one correspondence $(E,A,\varrho,v) \rightarrow (G,J,\varrho, \theta)$ between the longitudinal and torsional systems, where $v(\xi)=w(x)$ under the coordinate transformation $\xi^\prime(x)=1/A(x)$.

The development and systematic extension of transmutation operator methods can be traced through the frameworks established in \cite{Campos, Benitez, Benitez-II, Makovetsky, Dzarakhokhov, Pskhu}. For classical constant-coefficient and Sturm--Liouville type operators, the intertwining relation for the pair 
\[\textbf{L}=\dfrac{d^2}{dx^2}-q \quad \text{and} \quad  \textbf{M}=\dfrac{d^2}{dx^2}.
\]
is studied in the sense of distributions where $q$ is a locally integrable function \cite{Campos}. In this context, it is shown that a particular transmutation operator can be represented as a Volterra integral operator of the second kind. This structural mapping connects deeply to the concept of $L$-bases, which was first introduced in \cite{lbaseS-first} and later developed in \cite{lbases} to deepen the connection between transmutation operators and $L$-bases. This architecture serves as the direct foundation for the spectral parameter power series (SPSS) method developed in 
\cite{kravchenko2008representation, kravchenko2010spectral}. 

For differential operators involving an impedance framework, the connection between the pair
\[\textbf{L}_r=\dfrac{1}{r^2(x)}\dfrac{d}{dx}r^2 (x)\dfrac{d}{dx} \quad \text{and} \quad  \textbf{M}=\dfrac{d^2}{dx^2}
\]
is investigated in \cite{Benitez} under the assumption that $r \in C^1(I)$. It is shown that the transmutation operator 
\begin{equation} \label{trans-op} 
    \mathbf{T}u(x)=u(x)+\int_{-x}^x K(x,t)u'(t)\, dt,
\end{equation}
is bounded and invertible in this space, where the transmutation kernel satisfies the hyperbolic equation; and hence it is a transmutation operator for the pair $\textbf{L}_r$ and $\textbf{M}$. 

In a broader context, the analytical scope of transmutation operators theory extends to other operator classes and functional settings. For instance, in \cite{Makovetsky}, explicit governing equations for the kernels of transmutation operators within formally self-adjoint representation under the integral sign, which allows for the intertwining of Bessel equations with shifted spectral parameters. Furthermore, the scope of the transmutation has been extended to non-integer systems and explicit functional relations in \cite{Dzarakhokhov} by mapping respective eigenfunctions of integer-order derivatives, Bessel operators, and their fractional powers directly into one another. In a similar fractional setting, \cite{Pskhu} details the construction of transmutations that intertwine operators of fractional differentiation of distributed order, defined via a Lebesgue-Stieltjes measure to include both continuous and discrete parameters. 

The paper is organized as follows. Section \ref{section-2} outlines the necessary preliminaries and function spaces. In Subsection \ref{subsec:approximations}, we apply the classical method of successive approximations to prove the existence of a weak solution to the Goursat problem satisfied by the transmutation kernel $K$, see Proposition \ref{2.1}. Section \ref{section-3} introduces the notions of a standard basis and its associated standard transmutation operator for $\mathbf{L}_r$. Following the approach in \cite{Campos}, we establish that any bounded operator mapping the monomials $\{x^k\}$ into a standard $\mathbf{L}_r$-basis admits a regular integral representation and acts as a standard transmutation operator on $W^{3,1}(-a,a)$ (see Proposition \ref{prep:integral-representation}). Finally, Section \ref{section-4} reformulates the hyperbolic equation \eqref{2.4} satisfied by the transmutation kernel as an evolutionary system, demonstrating the existence and uniqueness of its solutions via Picard’s theorem.

\section{Existence of solutions for integral representations}\label{section-2}
Before establishing our main results, we first define the function spaces that serve as the framework for our analysis. Let $a>0$ and $I=(-a,a)$ be an open interval in $\mathbb{R}$. For $1 \leq p < \infty$, we denote by $L^p (I)$ the space of all Lebesgue measurable functions $f$ on $I$ for which 
\[
\int_I \lvert f(x)\rvert^p dx<\infty
\]
holds. The Sobolev space $W^{k,p}(I)$ is the subspace of functions $f \in L^p (I)$ whose weak derivatives up to order $k$ also belong to $L^p (I)$ where $k\in \mathbb{N}_0=\{0,1,2,3,\cdots\}$. Additionally, $C^k(I)$ denotes the space of functions that are continuously differentiable up to order $k$ on $I$. Expanding the derivative in \eqref{1.1-intro} yields the equivalent potential-form Sturm--Liouville equation
\begin{equation} \label{2.1} 
  -\dfrac{d^2u}{dx^2} + q(x)\dfrac{du}{dx} = \lambda u, \quad \text{where} \quad q(x) = -\dfrac{p'(x)}{p(x)}.
\end{equation}
In \eqref{1.1-intro}, the principal coefficient $p$ is assumed to be strictly non-vanishing on the closure $\overline{I}=[-a,a]$. For a potential function $q\in L^1(I)$ and a chosen reference point $x_0 \in \overline{I}$
integrating this relationship yields
\[
    p(x)=p(x_0) \exp \bigg(-\int^x_{x_0} q(s)\, ds\bigg).
\]
Note that for any scaling factor $\alpha \in \mathbb{C}$, the scaled coefficient $P=\alpha p$ generates the same differential equation and potential as $p$. Consequently, it is structurally convenient to fix a normalized coefficient satisfying $p(x_0)=1$. Under this normalization, the associated impedance function $r$ is defined as
\[
r(x):=exp \bigg(-\dfrac{1}{2} \int^x_{x_0} q(s)\, ds\bigg).
\]
Depending on the specific analytical setting, it is often more advantageous to formulate the problem in terms of the impedance function $r$ rather than the principal coefficient $p$. From this definition, the direct relationship between the impedance function of \eqref{1.1-intro} and the potential function of \eqref{2.1} is given by
\begin{equation} \label{2.2}
q(x):=-2 \dfrac{r'(x)}{r(x)}.
\end{equation}

We define the $p$-\textit{Wronskian} of $\varphi_0, \varphi_1\in W^{2,1}(-a,a)$ by
\begin{equation} \label{2.3}
W_p(\varphi_0, \varphi_1)(x) = \left| 
\begin{matrix} 
\varphi_0(x) & \varphi_1(x) \\ 
p \varphi_0'(x) & p \varphi_1'(x) 
\end{matrix} 
\right| 
= p(x) \left(\varphi_0(x)\varphi_1'(x) - \varphi_1(x)\varphi_0'(x)\right).
\end{equation}
If $p\equiv 1$, this definition reduces to the classical Wronskian $W(\varphi_0,\varphi_1)$. Differentiating \eqref{2.3} directly yields
\[
\dfrac{d}{dx}W_p(\varphi_0,\varphi_1)=(p\varphi_1')'\varphi_0-(p\varphi_0')'\varphi_1.
\]
In particular, if $\varphi_0$ and $\varphi_1$ are solutions of \eqref{2.1}, then $W_p(\varphi_0,\varphi_1)$ is constant.

Following the framework established in \cite{Benitez}, we consider the transmutation operator $\mathbf{T}$ given by \eqref{trans-op}, where the underlying kernel $K(x,t)$ satisfies the hyperbolic equation 
\begin{equation}\label{2.4} 
    \dfrac{\partial^2 K(x,t)}{\partial x^2}-\dfrac{\partial^2 K(x,t)}{\partial t^2}-q(x)\dfrac{\partial K(x,t)}{\partial x}=0, \quad |t|<x<a,
\end{equation}
subject to the Goursat boundary conditions 
\begin{equation}\label{2.6} 
    K(x,x)=1+Ae^{\frac{1}{2}\int_0^x q(s)\, ds},\quad K(x,-x)=Be^{\frac{1}{2}\int_0^x q(s)\, ds}
\end{equation}
with the parameter constraint $B=1+A$. 

By considering the Goursat problem \eqref{2.4}--\eqref{2.6} in the square domain $\overline{\mathcal{R}_a}=[-a,a]\times [-a,a]$ the change of variables $u=\frac{x+t}{2}$ and $v=\frac{x-t}{2}$ maps the governing equation into
\begin{equation}\label{2.7} 
    \dfrac{\partial^2 H(u,v)}{\partial u\partial v}=\frac{1}{2}q(u+v)\left(\dfrac{\partial H(u,v)}{\partial u}+\dfrac{\partial H(u,v)}{\partial v}\right).
\end{equation}
This equation is considered in the diamond-shaped domain $\Omega_a$ with vertices $(\pm a, 0)$ and $(0,\pm a)$. Under this transformation, the Goursat boundary conditions become
\begin{equation}\label{2.8} 
    H(u,0)=1+Ae^{\frac{1}{2}\int_0^x q(s)\, ds},\quad H(v,0)=Be^{\frac{1}{2}\int_0^x q(s)\, ds},
\end{equation}
where $B=1+A.$

Integrating the Goursat problem \eqref{2.7}--\eqref{2.8} yields the equivalent Volterra integral equation
\begin{equation} \label{2.9} 
\begin{aligned}
    H(u,v) = \,& A \exp \left\{\dfrac{1}{2} \int^u_0 q(s)\, ds \right\} + B \exp \left\{\dfrac{1}{2} \int^v_0 q(s)\, ds \right\} - A \\
    & + \dfrac{1}{2}\int^u_0 \int^v_0 q(\alpha+\beta) \left(\dfrac{\partial H (\alpha, \beta)}{\partial \alpha}  + \dfrac{\partial H (\alpha, \beta)}{\partial \beta}\right) d\beta \, d\alpha.
\end{aligned}
\end{equation}
\subsection{Method of successive approximations}\label{subsec:approximations}
The following propositions are essential for proving the uniqueness of the weak solution to the Goursat problem \eqref{2.4}--\eqref{2.6}. The first of these, Proposition \ref{prop. 2.1}, relies on the classical method of successive approximations \cite{Marchenko}.
\begin{proposition}\label{prop. 2.1}
Suppose that $q \in  L^2 (-a,a)$ such that the following conditions are satisfied
\[
\|q\|_{L^1(-a,a)} < 2 \quad \text{and} \quad \|q\|_{L^1(-a,a)}\|q\|_{L^2(-a,a)}<1.
\]
Then, the integral equation \eqref{2.9} has a unique solution $H(u,v) \in W^{1,1}(\Omega_a) \cap C(\Omega_a)$ that satisfies
\begin{equation} \label{2.10} 
\lvert H(u,v)\rvert \leq C,
\end{equation}
where $C$ is a constant depending on $\|q\|_{L^1(-a,a)}$ and $\|q\|_{L^2(-a,a)}$.
\end{proposition}
\begin{proof}
We seek a solution of the form
\begin{equation} \label{2.11} 
H(u,v)=\sum^\infty_{n=0} H_n (u,v)
\end{equation}
where 
\begin{equation} \label{2.12} 
H_0 (u,v)= A \exp \Bigl\{\dfrac{1}{2} \int^u_0 q(s)\, ds \Bigl\} +B \exp \Bigl\{\dfrac{1}{2} \int^v_0 q(s)\, ds\Bigl\}-A,
\end{equation} 
and, for $n\geq 0$,
\begin{equation} \label{2.13} 
H_{n+1} (u,v):= \dfrac{1}{2}\int^u_0 \int^v_0 q(\alpha+\beta) \mathcal{D} H_n (\alpha, \beta)\,d\beta d\alpha,
\end{equation}
with the operator $\mathcal{D}$ defined by 
\[\mathcal{D}f(u,v):= \dfrac{\partial f(u,v)}{\partial u}+\dfrac{\partial f(u,v)}{\partial v}, \quad f \in W^{1,1}(\Omega_a).\]

Using \eqref{2.12}, we compute
\[\mathcal{D}H_0 (u,v)=\dfrac{A}{2} q(u) \exp\Bigl\{\dfrac{1}{2} \int^u_0 q(s)\, ds \Bigl\} +\dfrac{B}{2} q(v) \exp\Bigl\{\dfrac{1}{2} \int^v_0 q(s)\, ds \Bigl\},
\]
which yields the upper bound
\begin{equation}\label{2.14}
\lvert \mathcal{D}H_0 (u,v)\rvert \leq \dfrac{|A|}{2} \lvert q(u)\rvert \exp\{Q_0 (u)\}+\dfrac{|B|}{2} \lvert q(v)\rvert \exp\{Q_0 (v)\},    
\end{equation}
where $Q_0 (x):= \dfrac{1}{2} \int^x_0 \lvert q(s)\rvert ds$ is an increasing function. 

Since $(u,v)\in\Omega_a$ implies $|u|\leq |u|+|v|\leq a$, and $Q_0$ is increasing, we obtain
\[
\begin{split}
\left|\exp\!\left\{\frac{1}{2}\int_0^u q(s)\,ds\right\}\right|
\leq \exp\!\left\{\frac{1}{2}\int_0^{|u|}|q(s)|\,ds\right\}= \exp\!\left\{Q_0(|u|)\right\}
\leq \exp\!\left\{Q_0(|u|+|v|)\right\}
\leq M,
\end{split}
\]
where $M := \exp\{\frac{1}{4}\|q\|_{L^1}\}$. An analogous bound holds for the corresponding exponential term in $v$. 

Keeping $B=1+A$ in mind and using the above bounds, we can continue estimating \eqref{2.14} as follows
\begin{equation}\label{2.17} 
    |\mathcal{D}H_0(u,v)|
    \leq \frac{M}{2}\Bigl(|A|\,\big(|q(u)|+|q(v)|\big) +\,|q(v)|\Bigr).
\end{equation}
Our aim is to prove the uniform convergence of the series in \eqref{2.11}, which will in turn imply the existence of the unique solution of \eqref{2.9}. To do this, it suffices to show that
\[
\lvert H(u,v) \rvert \leq \lvert H_0(u,v) \rvert+\sum_{n=1}^\infty \dfrac{M\big(2 |A|+1\big)}{2^{n+1}}\|q\|_{L^1}^{n-1}\Big(\|q\|_{L^1}^{2}+\big(2^{n-1}-1\big)a \|q\|_{L^2}^{n}\Big).
\]
For this purpose, we show the estimate
\begin{align}\nonumber
\lvert \mathcal{D}H_n(u,v)\rvert
&\leq \frac{M}{2^{n+1}}
\Bigg[\Big(|A||q(u)|+\big(|A|+1\big)|q(v)|\Big)\|q\|_{L^1}^n \\  \label{2.15} 
&\qquad + (2^n-1)\big(2|A|+1\big)\|q\|_{L^1}^{\,n-1}\|q\|_{L^2}^{\,n+1}\Bigg].
\end{align}
is valid for all $(u,v) \in \Omega_a$ and $n \in \mathbb{N}$. 
Observe that, when $n=0$ in \eqref{2.13}, we have
\[
H_1(u,v)=\dfrac{1}{2}\int_0^u \int_0^v q(\alpha+\beta)\mathcal{D}H_0(\alpha,\beta)\,d\beta\,d\alpha,
\]
and thus
\[
\mathcal{D}H_1(u,v)
= \dfrac{1}{2}\int_0^v q(u+\beta)\mathcal{D}H_0(u,\beta)\,d\beta
+ \dfrac{1}{2}\int_0^u q(\alpha+v)\mathcal{D}H_0(\alpha,v)\,d\alpha.
\]
Now, using the geometry of $\Omega_a$, we rewrite the integrals in a form convenient for estimating $\mathcal{D}H_1(u,v)$: 
\begin{align} \nonumber
    |\mathcal{D}H_1(u,v)|
    &\leq \frac{1}{2}\left|\int_0^{|v|} q(u+\operatorname{sgn}(v)\beta)\,\mathcal{D}H_0(u,\operatorname{sgn}(v)\beta)\, d\beta\right| \nonumber\\
    \quad &+\frac{1}{2}\left|\int_0^{|u|} q(\operatorname{sgn}(u)\alpha+v)\,\mathcal{D}H_0(\operatorname{sgn}(u)\alpha,v)\, d\alpha \right| \nonumber\\
    &\leq \frac{1}{2}\int_0^{|v|} \left|q(u+\operatorname{sgn}(v)\beta)\right| \left|\mathcal{D}H_0(u,\operatorname{sgn}(v)\beta)\right|\, d\beta \nonumber\\
    \quad &+\frac{1}{2}\int_0^{|u|} \left|q(\operatorname{sgn}(u)\alpha+v)\right| \left|\mathcal{D}H_0(\operatorname{sgn}(u)\alpha,v)\right|\, d\alpha. \label{2.16}
\end{align}
Substituting \eqref{2.17} into \eqref{2.16}, we obtain
\begin{align}
|\mathcal{D}H_1(u,v)|
&\leq \frac{M}{4}|A||q(u)| \int_0^{|v|} \left|q(u+\operatorname{sgn}(v)\beta)\right|\, d\beta \nonumber\\
&+ \frac{M}{4}(|A|+1)\int_0^{|v|} \left|q(u+\operatorname{sgn}(v)\beta)\right| \left|q(\operatorname{sgn}(v)\beta)\right|\, d\beta \nonumber\\
&+ \frac{M}{4}|A|\int_0^{|u|} \left|q(\operatorname{sgn}(u)\alpha+v)\right| \left|q(\operatorname{sgn}(u)\alpha)\right|\, d\alpha \nonumber\\
&+ \frac{M}{4}(|A|+1)|q(v)|\int_0^{|u|} \left|q(\operatorname{sgn}(u)\alpha+v)\right|\, d\alpha. \label{2.18}
\end{align}
We denote the four integral terms on the right-hand side of \eqref{2.18} by (I) through (IV), respectively. To bound these terms, we use the geometry of the diamond domain $\Omega_a$ which guarantees that for any $(u,v) \in \Omega_a$ the arguments  $u+\operatorname{sgn}(v)\beta$, $\operatorname{sgn}(v)\beta$, $\operatorname{sgn}(u) \alpha+v$, and $\operatorname{sgn}(u) \alpha$ remain within the interval $(-a,a)$. Assuming $q \in L^2 (-a,a)$, we estimate the first and fourth terms by a standard change of variables and direct integration
\[
(I)\leq \frac{M}{4}|A|\|q\|_{L^1}|q(u)|,
\]
and 
\[
(IV)\leq \frac{M}{4}\big(|A|+1\big)\|q\|_{L^1}|q(v)|.
\]
For the remaining terms (II) and (III), each integrand consists of a product of two functions depending on a single variable of integration. Applying the
Cauchy-Schwarz inequality and extending the integration limits to the full domain $(-a,a)$, we obtain
\begin{align*}
(II) &\leq \frac{M}{4}(|A|+1)
\left(\int_0^{|v|} |q(u+\operatorname{sgn}(v)\beta)|^2\,d\beta\right)^{1/2}
\left(\int_0^{|v|} |q(\operatorname{sgn}(v)\beta)|^2\,d\beta\right)^{1/2} \\
&\leq \frac{M}{4}(|A|+1)
\left(\int_{-a}^{a} |q(t)|^2\,dt\right)^{1/2}
\left(\int_{-a}^{a} |q(t)|^2\,dt\right)^{1/2} \\
&= \frac{M}{4}(|A|+1)\,\|q\|_{L^2}^2.
\end{align*}
An identical argument applied to (III) yields 
\[
(III) \leq \frac{M}{4}|A| \|q\|_{L^2}^2.
\]
Combining these estimates, we arrive at the final bound for $|\mathcal{D}H_1(u,v)|$:
\begin{align*} 
|\mathcal{D}H_1(u,v)|
&\leq \frac{M}{4}\|q\|_{L^1} \Bigl(|A|\,|q(u)| + (|A|+1)|q(v)|\Bigr) \\ \nonumber
&\quad + \frac{M(2|A|+1)}{4}\|q\|_{L^2}^2.
\end{align*}
which shows that \eqref{2.15} holds for $n=1$. If we assume that \eqref{2.15} holds for $n$, then it can be shown that it also holds for $n+1$:
\begin{align*}
\lvert \mathcal{D}H_{n+1}(u,v)\rvert
&\leq \frac{M}{2^{n+2}}
\Bigg[\Big(|A||q(u)|+(|A|+1)|q(v)|\Big)\|q\|_{L^1}^{\,n+1}\\
&\qquad + {(2^{n+1}-1)}\big(2|A|+1\big)\|q\|_{L^1}^{\,n}\|q\|_{L^2}^{\,n+2}\Bigg].
\end{align*}
Substituting this last estimate into \eqref{2.13} we have
\begin{align*}
    \lvert H_{n+1}(u,v)\rvert\leq &\dfrac{M}{2^{n+2}}\int_0^{|u|}\int_0^{|v|} \lvert q(\operatorname{sgn}(u)\alpha+\operatorname{sgn}(v)\beta)\rvert \Bigg[\big(\lvert A\rvert \lvert q(\operatorname{sgn}(u)\alpha) \rvert+\big(\lvert A\rvert+1\big) \lvert q(\operatorname{sgn}(v)\beta)\rvert\big) \|q\|_{L^1}^n\\
    &+\big(2^n-1\big)\big(2\lvert A\rvert +1\big)\|q\|_{L^1}^{n-1}\|q\|_{L^2}^{n+1}\Bigg]d\beta d\alpha\\
    =&I_1 +I_2 +I_3,
\end{align*}
where
\[
I_1=\dfrac{M\lvert A \rvert }{2^{n+2}}\|q\|_{L^1}^n \int_0^{|u|}\int_0^{|v|} \lvert q(\operatorname{sgn}(u)\alpha+\operatorname{sgn}(v)\beta) \rvert \,\lvert q(\operatorname{sgn}(u)\alpha) \rvert\, d\beta d\alpha,
\]
\[
I_2=\dfrac{M\big(\lvert A \rvert +1\big)}{2^{n+2}}\|q\|_{L^1}^n \int_0^{|u|}\int_0^{|v|} \lvert q(\operatorname{sgn}(u)\alpha+\operatorname{sgn}(v)\beta) \rvert \,\lvert q(\operatorname{sgn}(v)\beta) \rvert\,d\beta d\alpha,
\]    
and
\[
I_3=\dfrac{M\big(2^n-1\big)\big(2\lvert A\rvert +1\big)}{2^{n+2}}\|q\|_{L^1}^{n-1}\|q\|_{L^2}^{n+1}\int_0^{|u|}\int_0^{|v|} \lvert q(\operatorname{sgn}(u)\alpha+\operatorname{sgn}(v)\beta)\rvert \, d\beta d\alpha.
\]

We begin by evaluating the first integral, $I_1$. Write $s=\operatorname{sgn}(u)$ and $t=\operatorname{sgn}(v)$. Split the double integral by isolating the factor $|q(s\alpha)|$, which does not depend on $\beta$ as follows
\[
\int_0^{|u|}\int_0^{|v|} |q(s\alpha+t\beta)|\,|q(s\alpha)|\,d\beta\,d\alpha
= \int_0^{|u|} |q(s\alpha)|\left(\int_0^{|v|}|q(s\alpha+t\beta)|\,d\beta\right)d\alpha.
\]
For each fixed $\alpha\in[0,|u|]$, the substitution $y=s\alpha+t\beta$ yields
\[
\int_0^{|v|}|q(s\alpha+t\beta)|\,d\beta \;\le\; \int_{-a}^{a}|q(y)|\,dy \;=\; \|q\|_{L^1}.
\]
This bound does not depend on $\alpha$, so it factors out of the outer integral:
\[
\int_0^{|u|}\int_0^{|v|} |q(s\alpha+t\beta)|\,|q(s\alpha)|\,d\beta\,d\alpha \;\le\; \|q\|_{L^1}\int_0^{|u|} |q(s\alpha)|\,d\alpha.
\]
Applying the same substitution argument, now with $z=s\alpha \in (-a,a)$, to the remaining integral gives
\[
\int_0^{|u|}|q(s\alpha)|\,d\alpha \;\le\; \int_{-a}^{a}|q(z)|\,dz \;=\; \|q\|_{L^1}.
\]
Combining both bounds,
\[
\int_0^{|u|}\int_0^{|v|} |q(s\alpha+t\beta)|\,|q(s\alpha)|\,d\beta\,d\alpha \;\le\; \|q\|_{L^1}^{2},
\]
which yields
\[
I_1 \;\le\; \frac{M|A|}{2^{n+2}}\,\|q\|_{L^1}^{\,n+2}.
\]
Similarly, we have
\[
I_2\leq \dfrac{M \big(\lvert A \rvert +1\big)}{2^{n+2}}\|q\|_{L^1}^{n+2}.
\]
We next estimate $I_3$. For $(u,v)\in\Omega_a$ and $q\in L^1(-a,a)$, we bound the double integral directly. For fixed $\alpha$, the substitution $\gamma=s\alpha+t\beta$ gives
\[
\int_0^{|v|}|q(s\alpha+t\beta)|\,d\beta
\le \int_{-a}^{a}|q(\gamma)|\,d\gamma=\|q\|_{L^1(-a,a)},
\]
since the integrand is nonnegative and, because $(u,v)\in\Omega_a$, the image subinterval is contained in $(-a,a)$; in particular this bound does not depend on $|v|$. Integrating in $\alpha$ from $0$ to $|u|$,
\[
\int_0^{|u|}\int_0^{|v|}|q(s\alpha+t\beta)|\,d\beta\,d\alpha
\le |u|\,\|q\|_{L^1(-a,a)}
\le a\,\|q\|_{L^1(-a,a)}.
\]
Therefore,
\[
I_3\leq \dfrac{M\big(2^n-1\big)\big(2\lvert A\rvert +1\big)a}{2^{n+2}}\,\|q\|_{L^1}^{\,n}\|q\|_{L^2}^{n+1}.
\]
By combining the upper bounds computed for $I_1$, $I_2$, and $I_3$, we have for all $n\geq 0$
\begin{align*}
    \lvert H_{n+1}(u,v)\rvert\leq & \dfrac{M\big(2\lvert A\rvert +1\big)}{2^{n+2}}\|q\|_{L^1}^{n+2}\\
    &+\dfrac{M \big(2\lvert A\rvert +1\big)\big(2^n-1\big)a}{2^{n+2}}\|q\|_{L^1}^{n}\|q\|_{L^2}^{n+1}\\
    = &\dfrac{M\big(2\lvert A\rvert +1\big)}{2^{n+2}}\|q\|_{L^1}^{n}\bigg(\|q\|_{L^1}^{2}+\big(2^n-1\big)a\|q\|_{L^2}^{n+1}\bigg):=  M_{n+1}.
\end{align*}

To establish the global bound for $H(u,v)$, we sum the uniform estimates over $n$. Taking the absolute value on both sides of \eqref{2.11} and applying the triangle inequality yields 
\begin{equation} \label{sum}
\lvert H(u,v) \rvert \leq \sum_{n=0}^\infty \lvert H_n(u,v)\rvert.
\end{equation}
Substituting the estimate for $H_n(u,v)$ into \eqref{sum} yields:
\[
\lvert H(u,v) \rvert \leq (2|A|+1)M+|A|+\sum_{n=1}^\infty \dfrac{M\big(2 |A|+1\big)}{2^{n+1}}\|q\|_{L^1}^{n-1}\Big(\|q\|_{L^1}^{2}+\big(2^{n-1}-1\big)a \|q\|_{L^2}^{n}\Big).
\]
Distributing the term outside the brackets yields
\begin{equation} \label{H}
    \lvert H(u,v) \rvert \leq (2|A|+1)M+|A|+S_1+S_2,
\end{equation}
where 
\[
S_1=M\big(2 |A|+1\big)\sum_{n=0}^\infty \Bigg(\dfrac{\|q\|_{L^1}}{2}\Bigg)^{n+1}=\frac{M}{2}\big(2 |A|+1\big)\|q\|_{L^1}\sum_{n=0}^\infty \Bigg(\dfrac{\|q\|_{L^1}}{2}\Bigg)^{n},
\]
and 
\[
S_2=\dfrac{M\big(2 |A|+1\big)a}{4}\sum_{n=0}^\infty \|q\|_{L^1}^{n-1}\|q\|_{L^2}^{n}=\dfrac{M\big(2 |A|+1\big)a}{4\|q\|_{L^1}}\sum_{n=0}^\infty \Big(\|q\|_{L^1}\|q\|_{L^2}\Big)^{n}.
\]

$S_1$ converges absolutely if and only if $\|q\|_{L^1}<2$. Using the geometric series formula, we evaluate
\[
S_1=\frac{M}{2}\big(2 |A|+1\big)\dfrac{\|q\|_{L^1}}{1-\dfrac{\|q\|_{L^1}}{2}}=\dfrac{M\big(2 |A|+1\big)\|q\|_{L^1}}{2-\|q\|_{L^1}}.
\]
$S_2$ converges absolutely if and only if $\|q\|_{L^1}\|q\|_{L^2}<1$, using the geometric series formula gives the exact upper bound for $S_2$ as 
\[
S_2 \leq \dfrac{M\big(2 |A|+1\big)a}{4\|q\|_{L^1}\Big(1-\|q\|_{L^1}\|q\|_{L^2}\Big)}.
\]
Combining the explicit bounds for $S_1$ and $S_2$ back into \eqref{H}, we obtain
\begin{align*}
    \lvert H(u,v) \rvert \leq (2|A|+1)M+|A|+ \dfrac{M\big(2 |A|+1\big)\|q\|_{L^1}}{2-\|q\|_{L^1}}+\dfrac{M\big(2 |A|+1\big)a}{4\|q\|_{L^1}\Big(1-\|q\|_{L^1}\|q\|_{L^2}\Big)}:=C,
\end{align*}
where 
\[
C:= C \Big(M, |A|, \|q\|_{L^1}, \|q\|_{L^2}\Big).
\]
Since the series are bounded by a convergent sequence of numbers $M_n$, the Weierstrass $M$-test guarantees that the series \eqref{2.11} converges uniformly on $\Omega_a$, \eqref{2.10} holds and $H \in C(\Omega_a)$. Also, from \eqref{2.15}, allowing an identical argument of distributing the coefficients and splitting the summation into majorant geometric series, we obtain the global uniform bound for the total derivative as
\begin{equation} \label{DH_global_final}
|\mathcal{D}H(u,v)| \leq \frac{M \Bigl(|A||q(u)| + (|A|+1)|q(v)|\Bigr)}{2 - \|q\|_{L^1}} + \frac{M(2|A|+1)\|q\|_{L^2}}{2\|q\|_{L^1}(1 - \|q\|_{L^1}\|q\|_{L^2})}.
\end{equation}
From this last inequality, it follows that both the function and its generalized derivative are integrable, yielding $H \in W^{1,1}(\Omega_a).$ Combining this with out previous continuity result, we conclude that $H \in W^{1,1}(\Omega_a) \cap C(\Omega_a),$ which completes the proof. 
\end{proof}
\begin{proposition}\label{prop:continuity}
Let $q\in L^1(-a,a)$ such that $\|q\|_{L^1}<2$ and $\|q\|_{L^1}\|q\|_{L^2}<1$, and let $H$ be the unique solution of \eqref{2.9}. Then the mapping $q\mapsto H(u,v)$ is continuous from $L^1(-a,a)$ to $C(\Omega_a)$.
\end{proposition}
\begin{proof}
The proof follows along the lines of \cite[Prop.\ 20]{Campos} and establishes continuity by deriving an explicit bound on the difference between two solutions in terms of the $L^1$-distance between their corresponding potentials. 

Let $H(u,v)$ and $\tilde{H}(u,v)$ denote the unique solutions to \eqref{2.9} corresponding to the potentials $q, \tilde{q}\in L^1 (-a,a)$, respectively, where both potentials satisfy the constraints $\|\tilde{q}\|_{L^1}, \ \|q\|_{L^1}<2$ and $\|\tilde{q}\|_{L^1}\|\tilde{q}\|_{L^2}, \ \|q\|_{L^1}\|q\|_{L^2}<1$.

We define the difference function $M(u,v)=\tilde{H}(u,v)-H(u,v).$ Subtracting the integral equation \eqref{2.9} for $H(u,v)$ from that of $\tilde{H}(u,v)$, we find that $M(u,v)$ is the unique solution of the integral equation
\begin{equation} \label{eq:M_integral}
M(u,v)= M_0 (u,v)+\dfrac{1}{2}\int^u_0 \int^v_0 \tilde{q}(\alpha+\beta)\mathcal{D}M(\alpha,\beta)\, d\beta d\alpha
\end{equation}
where
\begin{align} \nonumber 
M_0 (u,v)= & A \Bigg(\exp \Bigl\{\dfrac{1}{2} \int^u_0 \tilde{q}(s)\,ds \Bigl\}-\exp \Bigl\{\dfrac{1}{2} \int^u_0 q(s)\, ds \Bigl\}\Bigg)\\ \nonumber 
&+B\Bigg(\exp \Bigl\{\dfrac{1}{2} \int^v_0 \tilde{q}(s)\,ds \Bigl\}-\exp \Bigl\{\dfrac{1}{2} \int^v_0 q(s)\, ds \Bigl\}\Bigg)\\ &+ \dfrac{1}{2}\int^u_0 \int^v_0 \big(\tilde{q}(\alpha+\beta)-q(\alpha+\beta) \big)\mathcal{D}M(\alpha,\beta)\, d\beta d\alpha. \label{M_00} 
\end{align}
Using the successive approximation argument from Proposition \ref{2.1}, the solution to \eqref{eq:M_integral} can be expressed as a uniformly convergent series
\[
M(u,v)= \sum_{n=0}^\infty M_n(u,v),
\]
where the base term is given by $M_0(u,v)$ and, subsequent iterations satisfy
\[
M_{n+1}=\dfrac{1}{2}\int^u_0 \int^v_0 \tilde{q}(\alpha+\beta)\mathcal{D}M_{n}(\alpha,\beta)\, d\beta d\alpha.
\]
By applying the same geometric series majorization as in the proof of Proposition \ref{2.1} to the potential $\tilde{q}$, we find

\begin{equation} \label{bound_M_0}  
\lvert M (u,v) \rvert \leq \|M_0\|_{C(\Omega_a)}+\frac{ \|M_0\|_{C(\Omega_a)} \|\tilde{q}\|_{L^1}}{2 - \|\tilde{q}\|_{L^1}} + \frac{a \|M_0\|_{C(\Omega_a)}}{4 \|\tilde{q}\|_{L^1} \big(1 - \|\tilde{q}\|_{L^1} \|\tilde{q}\|_{L^2} \big)}. 
\end{equation}
To bound $\|M_0\|_{C(\Omega_a)}$, we analyze the terms in \eqref{M_00}. Using the Lipschitz continuity of the exponential function, we start by bounding the first two terms in \eqref{M_00} as follows
\begin{align*}
  \Bigg|\exp \Bigl\{\dfrac{1}{2} \int^u_0 \tilde{q}(s)\,ds \Bigl\}-\exp \Bigl\{\dfrac{1}{2} \int^u_0 q(s)\, ds \Bigl\}\Bigg|&\leq \frac{1}{2}\exp 
  \Bigl\{\max \Bigl\{ \dfrac{1}{2} \int^u_0 \tilde{q}(s)\,ds, \dfrac{1}{2} \int^u_0 q(s)\, ds \Bigl\}\Bigl\}\, \|\tilde{q}-q\|_{L^1(-a,a)}\\
  &\leq \frac{1}{2}\exp 
  \Bigl\{ \dfrac{1}{2}\max \Bigl\{  \|q\|_{L^1} ,   \|\tilde{q}\|_{L^1}  \Bigl\}\Bigl\}\, \|\tilde{q}-q\|_{L^1(-a,a)}
\end{align*}
Analogously, for the second terms in \eqref{M_00}, we get
\begin{align*}
  \Bigg|\exp \Bigl\{\dfrac{1}{2} \int^v_0 \tilde{q}(s)\,ds \Bigl\}-\exp \Bigl\{\dfrac{1}{2} \int^v_0 q(s)\, ds \Bigl\}\Bigg|&\leq \frac{1}{2}\exp 
  \Bigl\{\dfrac{1}{2}\max \Bigl\{  \|q\|_{L^1} ,   \|\tilde{q}\|_{L^1}  \Bigl\}\Bigl\}\, \|\tilde{q}-q\|_{L^1(-a,a)}.
\end{align*}

For the third term in \eqref{M_00}, we use the boundedness of $\lvert \mathcal{D}M(u,v)\rvert$ on $\Omega_a$, as established in \eqref{DH_global_final}. Defining $C_{\mathcal{D}}:=\sup_{(\alpha, \beta)\in \Omega_a}\lvert \mathcal{D}M (\alpha, \beta)\rvert$, we get that
\begin{align*}
\left|\int^u_0 \int^v_0 \big(\tilde{q}(\alpha+\beta)-q(\alpha+\beta) \big)\mathcal{D}M(\alpha,\beta)\, d\beta d\alpha\right| &\leq
C_{\mathcal{D}}\int^u_0 \int^v_0 \big|\tilde{q}(\alpha+\beta)-q(\alpha+\beta) \big|\, d\beta d\alpha\\
&\leq C_{\mathcal{D}} a \|\tilde{q}-q\|_{L^1}(-a,a).
\end{align*}

Let us define $M_{\text{max}}:= \exp\Bigl\{\dfrac{1}{2} \max \big(\|q\|_{L^1}, \|\tilde{q}\|_{L^1}\big)\Bigl\}$. Combining these intermediate results and using the fact that $B=1+A$, we arrive at the following bound for $\|M_0\|_{C(\Omega_a)}$
\begin{equation} \label{M_0}
\|M_0\|_{C(\Omega_a)}\leq \Bigg(\dfrac{M_{\text{max}}}{2}\Big(2|A|+1\Big)+\dfrac{C_{\mathcal{D}}}{2}a\Bigg)\|\tilde{q}-q\|_{L^1(-a,a)}.
\end{equation}
Finally, substituting the bound \eqref{M_0} into \eqref{bound_M_0}, we conclude that there exists a positive constant
\[
C^* = C^*(|A|, a, \|q\|_{L^1}, \|\tilde{q}\|_{L^1})
\]
such that
\begin{equation*}
|\tilde{H}(u,v) - H(u,v)| \leq C^* \|\tilde{q}-q\|_{L^1}
\end{equation*}
uniformly for all $(u,v)\in \Omega_a$. Consequently, as $\|\tilde{q}-q\|_{L^1} \to 0$, we have $\|\tilde{H}-H\|_{C(\Omega_a)} \to 0$, which establishes the continuity of the solution mapping.
\end{proof}

\section{Representation results of the transmutation operator for SLEIF}\label{section-3}
Now we rewrite equation \eqref{1.1-intro} in terms of $r^2(x)=p(x)$ as follows
\begin{equation} \label{3.1} 
    \dfrac{d}{dx}\bigg(r^2(x)\dfrac{du}{dx}\bigg)+\lambda r^2(x)u=0, \ x \in I, \ \lambda \in \mathbb{C}.
\end{equation}
From now on, we assume $r(0)=1$ by choosing $x_0=0$. Then,
\begin{equation} \label{3.2}\
r(x)=exp \bigg(-\dfrac{1}{2} \int^x_{0} q(s)\, ds\bigg).  
\end{equation}
Notice that, if $q\in L^{1}(-a, a)$, then $r,r^{-1}\in L^{\infty}(-a,a)$. Indeed, absolute integrability bounds the exponent uniformly: $\left|\int_0^x q(t)\,dt\right| \le \int_{-a}^a |q(t)|\,dt =: C$. Hence $e^{-C/2} \le r(x)\le e^{C/2}$ for all $x \in (-a,a)$, keeping $r$ bounded both above and away from zero, and therefore $r^{-1}$ is bounded as well. Moreover, by \eqref{2.2}, if $q\in L^{1}(-a, a)$, then $r\in W^{1,1}(-a, a)$.

Next, consider the second-order differential operator
\begin{equation}\label{3.3} 
\mathbf{L}_r=\dfrac{1}{r^2(x)}\dfrac{d}{dx}r^2(x)\dfrac{d}{dx}.
\end{equation}
In \cite[Th.\ 11]{Benitez}, the transmutation property for the operators $\mathbf{L}_r$ and $\frac{d^2}{dx^2}$ in $C^3[-a,a]$ was established; namely,
\begin{align}\label{3.4} 
  \mathbf{L}_r \mathbf{T}_ru(x)=\mathbf{T}_r\dfrac{d^2}{dx^2}u(x),
\end{align}
where the transmutation operator is an integro-differential operator given by
\begin{align*}
      \mathbf{T}_ru(x)=u(x)-\int_{-x}^x K_r(x,t)u'(t)\, dt,
\end{align*}
and $K_r(x,t)=H(\frac{x+t}{2},\frac{x-t}{2})$ is the unique solution of \eqref{2.9} with $A=-1$, $B=0$.

In \cite[Rmk.\ 16]{Benitez}, using the transmutation relation \eqref{3.4}, it is observed that the images of $\{\varphi_k\}$ under the operator $\mathbf{L}_r$ defined in \eqref{3.3} satisfy the recursive relations
\begin{align}
  \mathbf{L}_r[\varphi_k]&=0, \qquad \qquad \qquad \quad \ k=0,1 \label{3.5}\\
  \mathbf{L}_r[\varphi_k]&=k(k-1)\varphi_{k-2}, \qquad k\geq 2\label{3.6}.
\end{align}
For the remainder of this section, we assume that the domain of $\mathbf{L}_r$ is $W^{3,1}(-a,a)$ and that $q\in L^1(-a,a)$. Following \cite{Campos, Benitez}, we introduce the following definition.
\begin{definition}
We say that a family of functions $\{\varphi_k\}$ satisfying \eqref{3.5} and \eqref{3.6} in the weak sense is an $\mathbf{L}_r$-base. Furthermore, we define:
\begin{enumerate}
    \item A \emph{standard} $\mathbf{L}_r$-base is an $\mathbf{L}_r$-base $\{\varphi_k\}$ such that $\varphi_k(0)=\varphi_k'(0)=0$ for all $k\geq 2$.   
    \item A \emph{standard transmutation operator} associated with $\mathbf{L}_r$ is a transmutation operator $T$ satisfying \eqref{3.4} on $W^{3,1}(-a,a)$ such the images of the monomials defined by $\varphi_k(x):=T[x^k]$ form a standard $\mathbf{L}_r$-base.
\end{enumerate}
\end{definition}

Now, let $\varphi_0$ and $\varphi_1$ be two linearly independent weak solutions of \eqref{3.5}. That is, 
\begin{equation}\label{3.7}
\varphi_i''=-2\frac{r'}{r}\varphi_i', \qquad i=0,1. 
\end{equation}
We introduce the modified kernel 
\begin{equation}\label{3.8}
G_p(x,s)=\frac{\varphi_0(s)\varphi_1(x)-\varphi_0(x)\varphi_1(s)}{W_p(\varphi_0,\varphi_1)},    
\end{equation}
where $W_p$ is the $p$-Wronskian defined in \eqref{2.3}. Note that when $p\equiv 1$, 
the kernel in \eqref{3.8} coincides with $G(x,s)$ used in \cite{Campos}.
\begin{proposition}\label{prop:recursive}
A standard $\mathbf{L}_r$-base $\{\varphi_k\}$ is completely determined by its first two elements. Moreover,
\begin{align}\label{3.9}
\varphi_k(x)=k(k-1)\int_0^{\int_0^x p(s)\, ds} G_p(x,s) \varphi_{k-2}(s)\, ds, \quad \forall k\geq 2,
\end{align}
where the kernel $G_p$ is defined in \eqref{3.8}.
\end{proposition}
\begin{proof}
Let $\Phi_k(t)$ denote the right-hand side of \eqref{3.9}. By computing the first and second derivatives of \eqref{3.8}, we obtain
\begin{equation}\label{3.10}
G_p'(x,s)=\frac{\varphi_0(s)\varphi_1'(x)-\varphi_0'(x)\varphi_1(s)}{W_p(\varphi_0,\varphi_1)},\quad
    G_p^{''}(x,s)=-2\frac{r'(x)}{r(x)}G'(x,s).
\end{equation}
Since $G_p(x,x)=0$, differentiating the right-hand side of \eqref{3.9} yields
\begin{align*}
\Phi_k'(x)=k(k-1)\int_0^{\int_0^x p(s)\, ds} G_p'(x,s)\, \varphi_{k-2}(s)\, ds.
\end{align*}
By \eqref{3.10}, it follows that 
\begin{equation} \label{3.11}
G_p'(x,x)=\dfrac{1}{p(x)}=\dfrac{1}{r^2 (x)}.
\end{equation}
Differentiating \eqref{3.9} once more, we obtain
\begin{align*}
\Phi_k''(x)&=k(k-1)\left(p(x)G_p'(x,x) \varphi_{k-2}(x)+\int_0^{\int_0^x p(s)\, ds} G_p''(x,s)\, \varphi_{k-2}(s)\, ds\right)\\
&=k(k-1)\left(\varphi_{k-2}(x)-\frac{2r'(x)}{r(x)}\int_0^{\int_0^x p(s)\, ds} G_p'(x,s) \varphi_{k-2}\, ds\right).
\end{align*}
Hence, $\Phi_k$ satisfies \eqref{3.6} and the initial conditions $\Phi_k(0)=\Phi_k'(0)=0$ for all $k\geq 2$. By uniqueness of the corresponding initial value problem for \eqref{3.6}, we conclude that \eqref{3.9} holds.
\end{proof}
Following \cite[Theorem 10]{Campos}, we derive the following integral representation for the operator $T$. 
\begin{proposition}\label{prep:integral-representation}
Let $\{\varphi_k\}$ be a standard $\mathbf{L}_r$-base, and let $T\colon W^{1,1}(-a,a)\to L^1(-a,a)$ be a bounded linear operator such that $T[x^k]=\varphi_k$ for all $ k\geq 0$. Then
\begin{enumerate}
\item[(i)] The operator $T$ admits the integral representation
\begin{equation}\label{3.12}
    T[u](x)=u(0)\varphi_0(x)+u'(0)\varphi_1(x)+\int_0^{\int_0^x p(s)\, ds} G_p(x,s)\, T[u''](s)\, ds,
\end{equation}
for all $u\in W^{3,1}(-a,a)$.
\item[(ii)] The operator $T$ is a standard transmutation operator associated with $\mathbf{L}_r$ on $W^{3,1}(-a,a)$.
\end{enumerate}
\end{proposition}
\begin{proof}
The proof of (i) follows from \cite[Th.\ 10]{Campos} using \eqref{3.9} and the density of polynomials in $W^{3,1}(-a,a)$. To prove (ii), we verify directly that $T$ satisfies the transmutation property \eqref{3.4}.  Since $G_p(x,x)=0$, differentiating the right-hand side of \eqref{3.12} yields
\begin{align*}
T'[u](x)=u(0)\varphi_0'(x)+u'(0)\varphi_1'(x)+\int_0^{\int_0^x p(s)\, ds} G_p'(x,s)\, T[u''](s)\, ds.
\end{align*}
Therefore, the second derivative of \eqref{3.12} is
\begin{align} \nonumber
T''[u](x)&=u(0)\varphi_0''(x)+u'(0)\varphi_1''(x)\\ \label{3.13} 
&+p(x)G_p'(x,x) T[u'']+\int_0^{\int_0^x p(s)\, ds} G_p''(x,s)\, T[u''](s)\, ds.
\end{align}
Substituting \eqref{3.11} into \eqref{3.13}, we get 
\begin{align*}
T''[u](x)&=u(0)\varphi_0''(x)+u'(0)\varphi_1''(x)\\
&+T[u'']+\int_0^{\int_0^x p(s)\, ds} G_p''(x,s)\, T[u''](s)\, ds.
\end{align*}
Next, using \eqref{3.7}  and the second equation in \eqref{3.10}, we obtain
\begin{align*}
T''[u](x)
=
-&2\frac{r'(x)}{r(x)}\Bigg(
u(0)\varphi_0'(x)+u'(0)\varphi_1'(x)
+\int_0^{x} G_p'(x,s)\,T[u''](s)\,p(s)\,ds
\Bigg)+T[u''](x).
\end{align*}
The expression in the large parentheses is precisely $T'[u]$. Therefore,
\[
T''[u](x)
=
-2\frac{r'(x)}{r(x)}T'[u](x)+T[u''](x),
\]
which is exactly the transmutation property \eqref{3.4} for the operator 
$\mathbf{L}_r$ defined in \eqref{3.3}. This completes the proof.
\end{proof}
Proposition \ref{prep:integral-representation} gives an integral representation for a standard transmutation operator $T$ associated with $\mathbf{L}_r$. We now study the convergence of a sequence of such operators. Let $q_n \subset L^1 (-a,a)$ be a sequence converging to $q \in L^1 (-a,a)$, and let $\{T_n\}$ be the corresponding sequence of standard transmutation operators. The following proposition shows that under suitable convergence assumptions, the limit operator $T$ is again a standard transmutation operator associated with $\mathbf{L}_r$.

\begin{proposition}\label{prop:convergence-uniformly}
Let $q_n\in L^1(-a, a)$, define $r_n$ from $q_n$ as in \eqref{3.2}, and let $\{T_n\}$ be a sequence of standard transmutation operators corresponding to the operator $\mathbf{L}_{r_n}$ defined in \eqref{3.3}. If $q_n\rightarrow q$ in $L^1(-a,a)$ and
\begin{equation*}
    T_n[x^k]\rightarrow T[x^k] \ \text{uniformly on }[-a,a] \ \text{ as }n\rightarrow \infty \ \text{for all} \ k\geq 0,
\end{equation*}
then $T$ is a standard transmutation operator on $W^{3,1}(-a,a)$ corresponding to $\mathbf{L}_r$. 
\end{proposition}
\begin{proof}
Let $\varphi_{k,n}(x)=T_n[x^k]$ and $\varphi_k=T[x^k]$. By Proposition \ref{prep:integral-representation}, it is enough to show that $\{\varphi_k\}$ forms a standard $\mathbf{L}_r$-base. 

Define
\begin{align*}
    s_n(x)&:=\varphi_{0,n}(x)+ \int_0^x (x-t) \frac{2 r_n'(t)}{r_n(t)}\varphi_{0,n}'(t)\, dt,\\
    s(x)&:=\varphi_{0}(x)+ \int_0^x (x-t) \frac{2 r'(t)}{r(t)}\varphi_{0}'(t)\, dt.
\end{align*}

Notice that $s_n''(x)=0$ if and only if $\mathbf{L}_{r_n}[\varphi_{0,n}]=0$. Therefore, $s_n(x)=c_1+c_2x$. Since $\{\varphi_{k,n}\}$ is a standard $\mathbf{L}_r$-base, we have 
\[
s_n(x)=\varphi_{0,n}(0)+\varphi_{0,n}'(0)x.
\]

By assumption, $\varphi_{0,n}=T_n[1]\rightarrow T[1]=\varphi_0$ uniformly on $[-a,a]$ as $n\rightarrow \infty$, and therefore 
\[
s(x)=\varphi_0(0)+\varphi_0'(0)x.
\]
Thus \( \mathbf{L}_r[\varphi_0]=0 \). Similarly, we obtain \( \mathbf{L}_r[\varphi_1]=0 \).

Let $G_{p_n}$ be the kernel corresponding to \( p_n=r_n^2 \), and let \( G_p \) be the kernel corresponding to \( p=r^2 \) for all $n$. Since \( q_n\to q \) in \( L^1(-a,a) \), it follows that \( p_n\to p \) in \( W^{1,1}(-a,a) \), with \( q_n=-p_n'/p_n \) and \( q=-p'/p \). Consequently, \( G_{p_n}(x,s)\to G_p(x,s) \) uniformly on $\mathcal{R}_a$.
To show that \( \{\varphi_k\} \) satisfies \eqref{3.6}, it is enough, by Proposition \ref{prop:recursive}, to verify \eqref{3.9}. Taking the limit in \eqref{3.9}, we obtain
\begin{align*}
\varphi_k(x)
&=\lim_{n\to\infty}\varphi_{k,n}(x)\\
&=\lim_{n\to\infty}k(k-1)\int_0^{\int_0^x p_n(s)\, ds} G_{p_n}(x,s)\,\varphi_{k-2,n}(s)\,p_n(s)\,ds\\
&=k(k-1)\int_0^{\int_0^x p(s)\, ds} G_p(x,s)\,\varphi_{k-2}(s)\,p(s)\,ds, \qquad k\ge 2.
\end{align*}
Hence \( \{\varphi_k\} \) is a standard \( \mathbf{L}_r \)-base.
\end{proof}
Now, using a limit argument, we generalize the existence result in \cite[Th.\ 4, Prop.\ 5]{Benitez} to include integrable potentials. In those results, the existence of a standard transmutation operator was shown for specific classes of potentials. 

By considering a sequence of continuous potentials $q_n$ converging to $q \in L^1(-a,a)$, we extend these result to the more general case. This approach not only preserves the key properties of the operator but also broadens the class of potentials for which the result holds, as shown in the following theorem.
\begin{theorem}
Under the assumptions of Proposition \ref{prop. 2.1}, there exists a unique weak solution $K \in W^{1,1}(\mathcal{R}_a)$ of \eqref{2.4} satisfying \eqref{2.6}. Moreover, the operator
\[
    \mathbf{T}u(x)=u(x)+\int_{-x}^x K(x,t)u'(t)\, dt,
\]
is a standard transmutation operator on $W^{3,1}(-a,a)$ corresponding to $\mathbf{L}_r$.
\end{theorem}
\begin{proof}
By Proposition \ref{prop. 2.1}, there exists a unique $H\in W^{1,1}(\Omega_a)$ satisfying \eqref{2.9}. 
Let $\{q_n\}\subset C^1(-a,a)$ be a sequence such that $q_n\rightarrow q$ in $L^1(-a,a)$. Since $q_n\in C^1(-a,a)$, there exists a sequence $\{K_n\}\subset C^2(\mathcal{R}_a)$ such that $K_n$ satisfies the following equation (see \cite[Prop.\ 5]{Benitez})
\begin{equation*}
    \dfrac{\partial^2 K_n(x,t)}{\partial^2 x}-\dfrac{\partial^2 K_n(x,t)}{\partial^2 t}-q_n(x)\dfrac{\partial K_n(x,t)}{\partial x}=0, \quad |t|<x<a.
\end{equation*}
In particular, for  all  $\varphi\in C_0^{\infty}(\mathcal{R}_a)$, $K_n$ satisfies
\begin{equation}\label{3.14} 
\int\int_{\mathcal{R}_a} \left((\partial_t K_n)\partial_t\varphi-(\partial_x K_n)\partial_x \varphi-q_n K_n \partial_x \varphi\right)\, dx dt=0.
\end{equation}
As before, let $\{H_n\}$ with $K_n(x,t)=H_n(\frac{x+t}{2},\frac{x-t}{2})$ be the solutions of the integral equation
\begin{align} \label{3.15}
H_n(u,v)=& A \exp \Bigl\{\dfrac{1}{2} \int^u_0 q_n(s)ds \Bigl\} +B \exp \Bigl\{\dfrac{1}{2} \int^v_0 q_n(s)ds\Bigl\}-A \\ \nonumber
    &+\dfrac{1}{2}\int^u_0 \int^v_0 q_n(\alpha+\beta) \bigg(\dfrac{\partial H_n (\alpha, \beta)}{\partial \alpha}+\dfrac{\partial H_n (\alpha, \beta)}{\partial \beta}\bigg)\,d \beta d \alpha.
\end{align}
By Proposition \ref{prop:continuity}, we have $H_n\rightarrow H$ uniformly on $\Omega_a$. Differentiating \eqref{3.15}, we obtain that $\partial_u H_n\rightarrow \partial_u H$ and $\partial_v H_n\rightarrow \partial_v H$ in $L^1(\Omega_a)$ as $n\rightarrow \infty$. Consequently, $\partial_x K_n\rightarrow \partial_x K$ and $\partial_t K_n\rightarrow \partial_t K$ in $L^1(\mathcal{R}_a)$. Finally, taking the limit in \eqref{3.14}, we conclude that $K$ is a weak solution of \eqref{2.4}. 

For the Goursat boundary conditions, we have
\begin{align*}
K_n(x,x)&=1+Ae^{\frac{1}{2}\int_0^x q_n(s)\, ds}\to 1+A e^{\frac{1}{2}\int_0^x q(s)\, ds}=K(x,x),\\
K_n(x,-x)&=Be^{\frac{1}{2}\int_0^x q_n(s)\, ds}\to Be^{\frac{1}{2}\int_0^x q(s)\, ds}=K(x,-x),
\end{align*}
as \( n\to\infty \). By Proposition \ref{prop:convergence-uniformly}, $\mathbf{T}$ is a standard transmutation operator because it is the uniform limit of the standard transmutation operators 
\[
T_nu(x)=u(x)+\int_{-x}^x K_n(x,t)\, u'(t)\, dt.
\] 
\end{proof}
The results above extend the theory of standard transmutation operators to integrable potentials.  In particular, we can relate any two standard transmutation operators on the space of polynomials, and we fill remaining gaps in the theory for general second-order linear differential operators.
\begin{remark}
Let $\mathcal{P}(\mathbb{R})$ denote the linear space of polynomials, and consider the operators
\[
P_{\pm}u(x)=\frac{u(x)\pm u(-x)}{2} \quad \text{and} \quad Au(x)=\int_0^x u(s)\, ds.
\]
Recall that the construction of standard transmutation operators for the Sturm--Liouville equation in impedance form is determined by their first two images, namely
$\varphi_0=T[1]$ and $\varphi_1=T[x]$. 
Thus, the explicit formula relating any two standard transmutation operators given in \cite[Prop.\ 18]{Campos} remains valid when restricted to $\mathcal{P}(\mathbb{R})$. That is, if $W(\psi_0,\psi_1)\ne 0$, then
\begin{align*}
T_{\varphi}=&\dfrac{W(\varphi_0,\psi_1)}{W(\psi_0,\psi_1)}T_{\psi}P_+-\dfrac{W(\varphi_0,\psi_0)}{W(\psi_0,\psi_1)}T_{\psi}AP_++\dfrac{W(\varphi_1,\psi_1)}{W(\psi_0,\psi_1)}T_{\psi}\dfrac{d}{dx}P_-\\
&-\dfrac{W(\varphi_1,\psi_0)}{W(\psi_0,\psi_1)}T_{\psi}P_-
\end{align*}
where $\{\varphi_k\}$ and $\{\psi_k\}$ are standard $\mathbf{L}_r$-base and $T_{\varphi}$ and $T_{\psi}$ are their corresponding standard transmutation operators. 
\end{remark}
\begin{remark}
Combining the results of \cite{Campos} with those presented in Section \ref{section-3}, we extend the theory of standard transmutation operators to more general second-order linear differential operators of the form
\[
a_2(x)\frac{d^2}{dx^2}+a_1(x) \frac{d}{dx}+a_3(x)
\] 
thereby addressing the gaps identified in the conclusions of \cite{Campos}.
\end{remark}
\section{Evolutionary equations}\label{section-4}

Evolutionary equations typically refer to partial differential equations that model the time evolution of a physical system. Such problems are expressed in the form
\begin{equation} \label{4.1} 
(\partial_t M(\partial_t) + A)U = F,
\end{equation}
where $U$ is the unknown solution, $A$ is an unbounded skew-self-adjoint operator, and $M$ is an analytic, bounded material law operator. For any $\nu\in \mathbb{R}$, we define 
\[
\mathbb{C}_{Re>\nu}:=\{z\in \mathbb{C}\colon \text{Re }z>\nu\}.
\]

Let $\mathcal{H}$ be a Hilbert space and $\nu\in \mathbb{R}$. Following \cite{Seifert}, we define the Hilbert space for the time derivative, denoted by $L^2_{\nu}(\mathbb{R},\mathcal{H})$,as the set of all Bochner measurable functions $f$ such that
\[
\int_{\mathbb{R}} \|f(t)\|_{\mathcal{H}}^2 e^{-2\nu t}\, dt<\infty.
\]
That is, $f \in L^2_{\nu}(\mathbb{R},\mathcal{H})$ if and only if the above integral is finite. Moreover, the space
$B(L^2_\nu)$ is defined as the set of continuous linear operators from $L^2_{\nu}(\mathbb{R},\mathcal{H})$ to itself. We will explain briefly how to define the time derivative on  $L^2_{\nu}(\mathbb{R},\mathcal{H})$. 

First, following \cite[Section 3.2]{Seifert}, let us consider the operator
\[
I_{\nu}\colon  L^2_{\nu}(\mathbb{R},\mathcal{H}) \to  L^2_{\nu}(\mathbb{R},\mathcal{H})
\] 
defined by
\begin{equation*}
 I_{\nu} f(t)=\begin{cases}
 \int_{-\infty}^t f(s)\, ds, \quad \nu>0,\\
 -\int_t^{\infty} f(s)\, ds, \quad \nu<0,
 \end{cases}
\end{equation*}
which is a bounded and injective operator. Thus, the \textit{time derivative}, denoted by $\partial_{t,\nu}$ is defined on $L^2_{\nu}(\mathbb{R},\mathcal{H})$ by $\partial_{t,\nu} f=I_{\nu}^{-1}f$, for any $\nu\ne 0$.
\begin{definition}
\cite{Seifert}
     Let $B(\mathcal{H})$ be the set of continuous linear operators on $\mathcal{H}$. A mapping $M: dom(M) \subset \mathbb{C} \rightarrow B(\mathcal{H})$ is called a material law if
    \begin{itemize}
        \item[(i)]  $dom(M)$ is open and $M$ is holomorphic,
        \item[(ii)]  there exists some $\nu \in \mathbb{R}$ such that $\mathbb{C}_{Re> \nu} \subset dom (M)$ and
        \begin{equation}
            \nonumber
            ||M||_{\infty, \mathbb{C}_{Re>\nu}} := \sup_{z\in \mathbb{C}_{Re>\nu}} ||M(z)|| < \infty.
        \end{equation}
    \end{itemize}
    Moreover, we define \textit{the abscissa of boundedness of} $M$ as
    \begin{equation}
        \nonumber
        s_b (M) := \inf \{ \nu \in \mathbb{R}; \ \mathbb{C}_{Re> \nu} \subset dom(M) \ and \  ||M||_{\infty, \mathbb{C}_{Re>\nu}} < \infty  \}.
    \end{equation}
    \end{definition}
Well-known examples of evolutionary equations include the heat equation, the wave equation, and the Schrödinger equation. For these equations, Picard's theorem (see \cite{Seifert}) has been applied to establish the existence of solutions. A very recent application of Picard's theorem is the proof of the existence of weak solutions for the angular-type Schrödinger equation \cite{macias2024}. 

The following theorem states Picard's theorem in the context of evolutionary equations, providing the fundamental well-posedness result for this class of problems.
\begin{theorem}\label{th:Picard}
\cite{Seifert}
    Let $\nu_0 \in \mathbb{R}$ and $\mathcal{H}$ be a Hilbert space. Suppose that $M:$ dom$(M) \subset \mathbb{C} \rightarrow B(\mathcal{H})$ is a material law with $s_b (M) < \nu_0$ and that $A:$ dom$(A) \subset \mathcal{H} \rightarrow \mathcal{H}$ is a skew-self-adjoint operator. Assume further that there exists a constant $C>0$ such that
    \begin{equation}\label{4.2}
        Re \langle \phi, zM(z)\phi \rangle_\mathcal{H} \geq C||\phi ||_\mathcal{H}^2 \ \ \ \ \ (\phi \in \mathcal{H}, z \in \mathbb{C}_{Re\geq \nu_0}).
    \end{equation}
     Then, for all $\nu \geq \nu_0$ the operator $\partial_{t, \nu} M(\partial_{t, \nu}) + A$ is closable and its closure has a bounded inverse
\begin{equation*}
    S_\nu := \overline{(\partial_{t, \nu}\,M(\partial_{t,\nu}) + A)}^{-1} \in \mathcal{B}(L_{2,\nu}(\mathbb{R}; \mathcal{H})).
\end{equation*}
Moreover, $S_\nu$ is causal and satisfies $||S_\nu||_{L(L_{2,\nu})} \leq 1/C$, and for every $F \in dom(\partial_{t, \nu})$ we have
\begin{equation*}
    S_\nu F \in dom(\partial_{t, \nu}) \cap dom(A).
\end{equation*}
Finally, if $\eta, \nu \geq \nu_0$ and $F \in L_{2, \nu}(\mathbb{R}; \mathcal{H}) \cap L_{2, \eta}(\mathbb{R}; \mathcal{H})$ then $S_\nu F = S_\eta F$.
\end{theorem}
We now show that the hyperbolic equation satisfied by the transmutation kernel \eqref{2.4} can be formulated as an evolutionary equation. To achieve this, we begin by considering the case of an absolutely integrable potential. Recall that $q \in L^1(-a,a)$, guarantees that both $r$ and $r^{-1}$ are bounded. Afterward, we apply Picard’s theorem to obtain an existence result.

\begin{theorem}\label{th:existence}
Let $q\in L^{1}(-a,a)$ be a real-valued potential. Then the hyperbolic equation \eqref{2.4} admits a solution $K$ such that $\partial_x K, \ \partial_t K \in L^2(\mathcal{R}_a)$.
\end{theorem}
\begin{proof}
Consider the Hilbert space $\mathcal{H} = (L^2(\mathcal{R}_a))^2$ equipped with the inner product
\[
\langle f, g\rangle_{\mathcal{H}} = \sum_{i=1}^2 \int_{\mathcal{R}_a} f_i \, \overline{g_i}.
\]

Define
\[
M(z)= \begin{pmatrix}
r^2 & 0\\
0 & 1/r^2 
\end{pmatrix}, \quad 
A =  \begin{pmatrix}
0 & -\partial_{x} \\
-\partial_{x} & 0 
\end{pmatrix}, \quad
U= \begin{pmatrix}\partial_t K \\ r^2 \partial_x K \end{pmatrix}, \quad 
F= \begin{pmatrix}0 \\ 0 \end{pmatrix}.
\]
We now verify that the differential equation satisfied by the transmutation kernel \eqref{2.4} can be rewritten as an evolutionary equation \eqref{4.1}. Indeed,
\[
(\partial_tM(\partial_t)+A)U=\begin{pmatrix}
r^2\partial_t & -\partial_x\\
-\partial_x & \partial_t/r^2
\end{pmatrix}
\begin{pmatrix}
\partial_t K\\
r^2 \partial_x K
\end{pmatrix}
=
\begin{pmatrix} 
r^2(\partial_t^2K-\partial_x^2K+q\partial_xK)\\
-\partial_x\partial_t K + \partial_t\partial_x K
\end{pmatrix}
=
\begin{pmatrix}
0\\
0
\end{pmatrix}.
\]

By \cite[Ex.\ 5.3.1]{Seifert}, $M$ is a material law with $s_b(M)=-\infty$. Moreover, since $(\partial_x)^* = -\partial_x$, the operator $A$ is skew-self-adjoint.

Let $z \in \mathbb{C}_{\mathrm{Re}\, z \geq \nu_0}$ and $\phi=(\phi_1,\phi_2) \in \mathcal{H}$. Using that $r,r^{-1}\in L^{\infty}(-a,a)$, then we verify that the coercivity condition \eqref{4.2} holds
\[
\begin{aligned}
\mathrm{Re}\, \langle \phi, zM(z)\phi \rangle  
&= \mathrm{Re}\, \langle (\phi_1, \phi_2), (zr^2\phi_1, zr^{-2}\phi_2) \rangle_{\mathcal{H}} \\
&= \mathrm{Re}\left(\langle \phi_1, zr^2\phi_1 \rangle_{L^2} + \langle \phi_2, zr^{-2}\phi_2\rangle_{L^2}\right)\\
&=\mathrm{Re}\left(\langle r\phi_1, zr\phi_1 \rangle_{L^2} + \langle r^{-1}\phi_2, zr^{-1}\phi_2\rangle_{L^2}\right)\\
&=\mathrm{Re}(z)  \left(\|r\phi_1\|_{L^2}^2+\|r^{-1}\phi_2\|_{L^2}^2\right)\\
&\geq \nu_0 \min\{ \|r^{-1}\|^{-2}_{L^{\infty}}, \|r\|^{-2}_{L^{\infty}}\} \left(\|\phi_1\|_{L^2}^2+\|\phi_2\|_{L^2}^2\right),
\end{aligned}
\]
where we used that $\mathrm{Re}\, z \geq \nu_0$.
This establishes coercivity of $zM(z)$ uniformly on the whole
half-plane $\mathbb C_{\operatorname{Re}z\ge\nu_0}$, as required by
Theorem \ref{th:Picard}.

By Theorem \ref{th:Picard}, it follows that the homogeneous evolutionary equation
$(\partial_tM(\partial_t)+A)U=F$ admits a unique solution $U\in
\mathcal H$. In particular, there exists a solution $K$ of \eqref{2.4} such
that $\partial_xK,\,r^2\partial_tK\in L^2(\mathcal R_a)$. By the boundedness of $r$, we get that $\partial_xK, \partial_tK\in L^2(\mathcal R_a)$ as we desired.

\end{proof}

\section{Conclusion}
In this work, we establish the existence of a transmutation operator for the Sturm--Liouville equation in impedance form on $W^{3,1}(-a,a)$ for potentials $q\in L^1(-a,a)$. We also extend the theory of standard transmutation operators developed in \cite{Campos}. As a future task, we will attempt to remove the smallness condition from Proposition \ref{2.1} and generalize the result to locally integrable potentials. This generalization broadens the applicability of the framework to a wider class of second-order differential operators, helping bridge a gap in the theory and enabling further analytical and computational developments in both direct and inverse spectral problems.


\end{document}